\documentclass[12pt]{article}

\usepackage{arxiv}

\usepackage[utf8]{inputenc} 
\usepackage[T1]{fontenc}    
\usepackage[hidelinks]{hyperref}       
\usepackage{url}            
\usepackage{booktabs}       
\usepackage{amsfonts, amsmath, dsfont}       
\usepackage{amsthm}
\usepackage{nicefrac}       
\usepackage{microtype}      
\usepackage[numbers]{natbib}
\usepackage{doi}
\usepackage{fancyhdr, graphicx, float, geometry}
\usepackage{enumitem}
\usepackage{tikz}
\usepackage{subcaption}

\newtheorem{remark}{Remark}
\newtheorem{theorem}{Theorem}
\newtheorem{lemma}{Lemma}

\newcommand{\rr}{\mathbb{R}}
\newcommand{\rd}{\mathrm{d}}
\newcommand{\rds}{\mathrm{d}s}
\newcommand{\rdv}{\mathrm{d}v}
\newcommand{\rdt}{\mathrm{d}t}
\newcommand{\rdx}{\mathrm{d}x}
\renewcommand{\L}{\mathcal{L}}
\DeclareMathOperator{\supp}{supp}
\renewcommand{\ss}{\mathbb{S}}
\newcommand{\pp}{\mathbb{P}}
\newcommand{\eps}{\varepsilon}
\newcommand{\<}{\left\langle}
\renewcommand{\>}{\right\rangle}

\DeclareMathOperator*{\esup}{ess\,sup}

\title{Kinetic chemotaxis tumbling kernel determined from macroscopic quantities}

\author{ \href{https://orcid.org/0000-0002-5090-9218}{\includegraphics[scale=0.06]{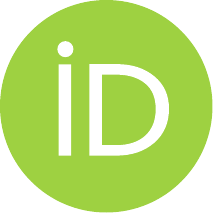}\hspace{1mm}Kathrin Hellmuth}\\
	Department of Mathematics\\
	University of W\"urzburg\\
	97074 Würzburg, Germany\\
	\texttt{kathrin.hellmuth@mathematik.uni-wuerzburg.de}\\
	\And 
	\href{https://orcid.org/0000-0003-2033-8204}{\includegraphics[scale=0.06]{orcid.pdf}\hspace{1mm}Christian Klingenberg} \\
	Department of Mathematics\\
	University of W\"urzburg\\
	97074 Würzburg, Germany\\
	\texttt{klingen@mathematik.uni-wuerzburg.de}\\
	\And 
	\href{https://orcid.org/0000-0001-9210-8948}{\includegraphics[scale=0.06]{orcid.pdf}\hspace{1mm}Qin Li} \\
	Department of Mathematics\\
	University of Wisconsin-Madison\\
	Madison, WI, 53705, USA\\
	\texttt{ qinli@math.wisc.edu}\\
	\And 
	\hspace{1mm}Min Tang \\
	School of Mathematics\\
	Shanghai Jiaotong University\\
	Shanghai, 200240, China\\
	\texttt{ tangmin@sjtu.edu.cn}
}

\begin{document}
	\maketitle
	\begin{abstract}
		Chemotaxis is the physical phenomenon that bacteria adjust their motions according to chemical stimulus. A classical model for this phenomenon is a kinetic equation that describes the velocity jump process whose tumbling/transition kernel uniquely determines the effect of chemical stimulus on bacteria. The model has been shown to be an accurate model that matches with bacteria motion qualitatively. For a quantitative modeling, biophysicists and practitioners are also highly interested in determining the explicit value of the tumbling kernel. Due to the experimental limitations, measurements are typically macroscopic in nature. Do macroscopic quantities contain enough information to recover microscopic behavior? In this paper, we give a positive answer. We show that when given a special design of initial data, the population density, one specific macroscopic quantity as a function of time, contains sufficient information to recover the tumbling kernel and its associated damping coefficient. Moreover, we can read off the chemotaxis tumbling kernel using the values of population density directly from this specific experimental design. This theoretical result using kinetic theory sheds light on how practitioners may conduct experiments in laboratories.
	\end{abstract}
	\keywords{kinetic chemotaxis equation; velocity jump process; singular decomposition;  unique reconstruction;  tumbling kernel}
	\section{Introduction}
	
	Bacteria and microorganisms can move autonomously and react to external stimuli, such as food or danger,  which is an important factor in  evolution. If the movement is affected by a chemical stimulus, this phenomenon is called chemotaxis. Chemotaxis phenomena are widely observed among motile organisms and particularly well studied for  Escherichia coli (E.coli) cells. When the bacterial movement consists of two alternating phases in which they either run along a straight line or re-orient by changing the direction of travel (tumbling), their movement is called a velocity jump process. 
	
	The kinetic chemotaxis model describes this behaviour statistically \cite{AltChemotaxis,OthmerHillen2Chemo, chalub2004kinetic,ErbanOthmerKineticChemotaxis}:
	\begin{alignat}{4}
		&\frac{\partial }{\partial t} f(x,t,v) + v\cdot \nabla_x f(x,t,v) = \L(f){(x,t,v)} - \sigma{(x,v)} f{(x,t,v)}\,\label{eq:chemotaxis}\\
		&f(x,t=0,v) = \phi(x,v)\label{eq:initialdata}
	\end{alignat}
	The equation describes the evolution of the  population density of bacteria $f(x,t,v)$ on the phase space $(x,v)\in \rr^d\times V$, $d\in \{2,3\}$ during the time interval $t\in[0,T]$ with initial condition $\phi$. Experimental data suggests that bacteria move at a constant speed, and we set $V:=\mathbb{S}^{d-1}$ to reflect this fact. The left side of equation \eqref{eq:chemotaxis} describes the motion of the bacteria moving along a straight line in direction $v$ from location $x$. The two terms $\L(f)$ and $\sigma f$ on the right hand side of the equation describe the velocity jump process arising from the reorientation by tumbling. In particular,
	\begin{equation}\label{eqn:def_L}
		\L(f)(x,t,v) = \int_V K(x,v,v') f(x,t,v')\rd{v'}\,,
	\end{equation}
	collects the particles that change their velocity from $v'$ to $v$, and
	\begin{equation}\label{eqn:def_sigma}
		\sigma {(x,v)}= \int_V K(x,v',v) \rd{v'}\,,
	\end{equation}
	describes the ratio of particles changing velocity from $v$ to others, and thus disappearing in a statistical sense from the phase point $(x,t,v)$. As such, these two terms are called gain and loss terms respectively.
	Let us mention that in certain applications, bacteria  generate a self-attracting/propulsing stimulus whose effect is then included in the tumbling kernel $ K(x,v,v')$ through the ``concentration'' term, see chemotaxis modeling~\cite{Bournaveas_BlowUpKineticChemotaxis,BournaveasPoissonModel,chalub2004kinetic, chalub2006PreventOvercrowdingKineticChemotaxis}. In this paper we consider this a next order concern, and set $K$ to be a fixed function in space such that our model coincides with the linear Boltzmann equation. We discuss more details in the Conclusion \ref{sec:conclusions}.
	
	The tumbling kernel $ K(x,v,v')$ encodes the probability of bacteria changing velocity directions by tumbling. The value of this kernel is affected by the chemoattractant density, usually denoted by $c(x)$. To make the discussion more clear and concise, we only consider the case when the chemoattractant concentration $c(x)$ is given. This indicated that the chemical stimulus cannot be consumed or produced by the bacteria and $K$, $\sigma$ are independent of time.
	
	Different types of bacteria take different values of $K$ and $\sigma$ and are differently affected by the  concentrations of the chemical stimulus (chemoattractant). 
	Since the tumbling kernel $K$ and the damping factor $\sigma$ uniquely determine the law of the bacterial motion in \eqref{eq:chemotaxis}, biologists and practitioners are highly interested in identifying them for future motion predictions, see important applications in bioreactors~\cite{Singh2008ApplicationChemotaxisSoilRemediation},  the spread and prevention of diseases~\cite{JOSENHANS2002605} and biofilm formation~\cite{biofilm}.
	
	To identify $K$ and $\sigma$ in practice, experiments are conducted to measure observables of bacterial behaviour so to indirectly infer the tumbling coefficient. The practical difficulty is that measuring the time dynamical data of velocity dependent bacteria density $f(x,t,v)$  is not always feasible. One would have to trace the trajectory of each single bacterium for a long time, which is technically difficult when there are a lot of bacteria \cite{measureVelocityDependentf}. Instead, the time evolution of the macroscopic quantities such as the density 
	\[
	\rho(x,t):=\<f\> = \int_V f\, \rdv
	\]
	is much easier to obtain by counting bacteria on a time series of photos, see also other more sophisticated methods~\cite{measureRho}. This poses an interesting mathematical question: can the macroscopic measurements on bacteria density, as a function of time, uniquely determine the values of $K$ and $\sigma$? 
	
	At the first sight, the answer should be negative. Indeed, the to-be-inferred parameters are functions posed on the microscopic level and have $v$ dependence, while the measurements are purely on the macroscopic level with $v$ dependence eliminated. This mismatch leads to some mathematical difficulty, to overcome which, a mechanism that triggers information on the microscopic level is needed. We introduce this mechanism by examining the time dependence, and playing with the singularity in the initial data. It turns out that if one places a special set of singularity in the initial data and introduces the corresponding singularity to the measuring operator that is concentrated at the compatible time and location, we can prove that the coefficients $K$ and $\sigma$ are uniquely reconstructable. Furthermore, the values can be directly read off from the measurements.

	The mathematical machinery that allows us to explicitly express the reconstruction is a technique termed singular decomposition. It is a technique that is specifically designed for studying inverse problems from kinetic theory, and has been traditionally used to investigate stationary radiative or photon transport equation, see~ \cite{ Bal_SingDecomp_2009,Ball_StabilitySingDecomp_2010, Ball_SingDecomp_2018,Ball_SingDecomop_2012,Choulli_1996,LiSun_SingularDecomp, Ren_ReviewNumericsTransportBasedImaging, StefanovZhongSingDecomp} for instance. The most classical use of the singular decomposition allows the data to be $v$ dependent, but the measuring location is typically set only to be on the boundary. Difficulties are introduced when only velocity independent measurements are available~\cite{Bal_angularAveragedMeasurementsInstability, LiWangChen_BayesianInstabilityAngularAveragedMeasurements, Ren_ReviewNumericsTransportBasedImaging, ZhaoZhong_angularlyAveragedIllposed}. In this setting, one no longer has the luxury of the freedom from the velocity dependence. In our project, however, we use measurements in time, containing information from the interior of the domain. The main task in this project is to investigate if these freedom could counter the difficulties induced by the loss of velocity information. It turns out from our study that the availability of short-time data is also crucial. In both the reconstruction of $\sigma$ and $K$, we heavily rely on the design of initial/measuring time and locations that precisely reflects the parameters to be reconstructed.
	
	Using measurements to identify bacteria motion is of high practical interests to biologists. However, even though chemotaxis and inverse problems are both very active areas of research, largely hindered by the lack of rigorous mathematical tools, very little is known theoretically if the experiments can truly reflect bacteria chemotaxis behavior. In practice, the most popular parabolic Keller-Segel model is on the macroscopic level, and it is common for practitioners to assume a heuristically obtained parametrized form for the model coefficients and estimate only these parameters by experimental data \cite{giometto2015generalized, ChemotaxisRecoverDGamma, RecoverChemotacticSensivity,ChemotacticSensitivityVariesALot}. Numerically, one can study the identifiability  of the  chemoattractant sensitivity for the macroscopic models, as in \cite{FisterMcCarthy_UniqueChemotacticSensitivity_2008, EggerPietschmannSchlottbohm_UniqueChemotacticSensitivity_2015}, where it was shown how the reconstruction from the regularized problem converges to the true solution as the noise vanishes.
	
	As the techniques such as kinetic theory and singular decomposition ripen, we are convinced that these applications can be re-examined afresh, with a more rigorous viewpoint. It is our aim to prove the unique reconstructability of the kinetic tumbling kernel $K$ and loss coefficient $\sigma$ using only the macroscopic measurements. Hopefully these arguments provide foundations to the algorithms that execute the reconstruction in reality.

	
	The article is structured as follows: In section \ref{sec:ProblemSetup}, we provide the problem setup. Sections \ref{sec:RecSigma} and \ref{sec:RecK} build the heart of this article and contain the proofs of the unique reconstructability of $\sigma$ and $K$ respectively. For both cases, the singular decomposition technique will be used for carefully prepare initial data and measurement test function. The article is concluded by Section \ref{sec:conclusions}.
	\section{Problem setup}\label{sec:ProblemSetup}
	The setup of the lab experiment is connected to the mathematical formulation in this section.
	
	In the lab experiment, an environment with a fixed chemical concentration is prepared, so $K$ and $\sigma$ can be thought as constants in time. Initially, bacteria are placed in this environment in a controlled manner and the evolution  of the macroscopic bacteria density along time is captured locally in time and space by measurements.
	
	Mathematically,~\eqref{eq:chemotaxis} is considered the model equation for bacteria motion. Though  the plate that bacteria are supported in is finite and thus provides certain boundary conditions, for  simplicity, we assume an infinitely big domain where  boundary poses no effect to the dynamics. We should note, however, that the inversion mechanism that to be employed in this paper makes use of compactly supported initial condition and data-measurement at small time, so potential boundary conditions are not expected affect the reconstruction. Hence we expect the inversion mechanism to be extendable to the case where the interior of a finite domain problem is considered. 
	
	The initial condition can be controlled, and we prescribes it as $\phi(x,v)$. Let us mention that    singularity in $v$ domain can be realized through suitable experimental apparatuses, for instance, the bacteria can be confined in a thin pipe and released   into the environment, see for example~\cite{ThinLayerSwarmingMotility} where E.coli bacteria were examined, and ~\cite{SyntheticMicroswimmers}  where the authors manipulated synthetic microswimmers through micro-confinement. The algae Euglena gracilis can be controlled by polarized light, which was exploited by the authors of~\cite{ControlValgae} and the references therein.
	
	From now on, we present the analysis for dimension  $d=3$, but  methods and results can be extended to deal with $d=2$ dimensions  as well.
	
	For every given initial data $\phi(x,v)$,
	we denote the solution to the PDE~\eqref{eq:chemotaxis} equipped with initial data~\eqref{eq:initialdata} by $f_\phi$, and the macroscopic quantity is
	\begin{equation*}
		\rho_{\phi}(x,t):=\<f_\phi\>=\int_V f_\phi\,\rd v\,.
	\end{equation*}
	This builds the following map:
	\[
	\Lambda_K:\phi\quad\to\quad \rho_\phi(t,x)\,.
	\]
	
	To be more compatible with the real practice, for each detector, we let $\psi(x,t)\in C^{\infty}_c$  present its profile, then the detector's reading would be $\rho_\phi$ tested on this test function, the output of the following measurement operator:
	\begin{equation}\label{eqn:measurement}
		M_{\psi}\left(f_\phi\right)=M_{\psi}\left(\rho_\phi\right)= \int_0^T\int_{\mathbb{R}^3}\rho_{\phi} \, \psi(x,t)\,\rd x\,\rd t \,.
	\end{equation}
	Since the measurement operator only acts on the density $\rho_\phi$, we abuse the notation and let $M_{\psi}\left(f_\phi\right)=M_{\psi}\left(\rho_\phi\right)$ when $\rho_\phi=\<f_\phi\>$.
	
	
	It is immediate that for every fixed $\psi$, the measurement is the one instance of reading of $\Lambda_K[\phi]$:
	\[
	M_{\psi}\left(f_{\phi}\right) = \int_{0}^T\int_{\rr^3}\psi(t,x)\Lambda_K[\phi](t,x)\,\rd{t}\,\rd{x}\,.
	\]
	
	We claim that $\Lambda_K$ encodes all the needed information to uniquely recover $\sigma$ and $K$, and the reconstruction process depends on the special design of $\phi$ and $\psi$, namely:
	
	\begin{theorem}
		Under mild conditions, one can uniquely reconstruct $\sigma$ and $K$ using the map $\Lambda_K$. Moreover, by properly choose $(\phi,\psi)$, the reconstruction can be explicit using the reading of $M_{\psi}(f_\phi)$.
	\end{theorem}
	
	Throughout the paper we assume $\sigma$ and $K$ are time-independent, and the admissible sets are:
	\begin{align*}
		&\mathcal{A}_{\sigma} = \{\sigma \in C_{+}(\rr^3\times V)\mid \|\sigma\|_{\infty}\leq C_{\sigma}\}\\
		&\mathcal{A}_{K} = \{K \in C_{+}(\rr^3\times W)\mid \|K\|_{\infty}\leq C_{K}\}
	\end{align*}
	where we set $W:= \{(v,v')\in V\times V\mid v\neq v'\}$. The reconstruction procedure is performed on these sets.
			
	
	
	\section{Reconstructing $\sigma$}\label{sec:RecSigma}
	We dedicate this section for reconstructing $\sigma(x,v)$, and showing the following theorem.
	
	\begin{theorem}[Unique reconstruction of $\sigma$]\label{thm:unique_sigma}
		Let $\sigma\in\mathcal{A}_{\sigma}$ and $K\in\mathcal{A}_{K}$. The map $\Lambda_K$ uniquely determines $\sigma(x,v)$. In particular, for any $(x,v)$, by a proper choice of $\phi$ and $\psi$, one can explicitly express $\sigma(x,v)$ in terms of $M_\psi(\rho_\phi)$, with $\rho_\phi$ being the density associated with $f_\phi$ that solves~\eqref{eq:chemotaxis}.
	\end{theorem}
	\begin{remark}\label{rem:recSigmaSmallt}
		We note the statement of the result can be extended to treat time-dependent $\sigma$ as well. To recover $\sigma(x,t,v)$ at a particular time $t$-horizon, the data $\phi$ needs to be provided, and the measurements need to be collected close enough to the $t$. As the following proof shows, if $\phi$ is provided at $t_0>t-C$ with $C=(|V|C_K)^{-1}$ for $C_K$ being the bound from the admissible set, the time-dependence of $\sigma$ can be reconstructed as well.
		
		
	\end{remark}
	
	The main technique used in the proof is termed the singular decomposition developed in~\cite{Choulli_1996}, and then extensively used in other following works~\cite{  Bal_SingDecomp_2009,Ball_StabilitySingDecomp_2010,Ball_SingDecomp_2018,Ball_SingDecomop_2012,LiSun_SingularDecomp, StefanovZhongSingDecomp}. The idea is to design a special set of sources $\phi$ that introduces singularity to the solution, in short time, the singularity is mostly preserved along the propagation trajectory. By properly choosing the compatible $\psi$, the singular information can be picked up by the measurements. Mathematically, to identify the singular component of the solution, we decompose $f_\phi$ into parts that exhibit different regularity. In particular, we decompose $f$ into
	\[
	f_\phi=f_{\phi,0}+f_{\phi,\geq 1}
	\]
	where $f_{\phi,0}$ and $f_{\phi,\geq1}$ solve the following equations respectively:
	\begin{equation}\label{eqn:f0}
		\left\{\begin{array}{rl}\partial_t f_{\phi,0} + v\cdot \nabla f_{\phi,0} &= -\sigma f_{\phi,0},\\
			f_{\phi,0}(x,t=0,v)&= \phi(x,v),\end{array}  \right.
	\end{equation}
	and
	\begin{equation}\label{eqn:fgeq1}
		\left\{ \begin{array}{rl}\partial_t f_{\phi,\geq1} + v\cdot \nabla f_{\phi,\geq1} &= -\sigma f_{\phi,\geq1} + \L(f_{\phi,0}+f_{\phi,\geq1}),\\
			f_{\phi,\geq1}(x,t=0,v) &= 0.\end{array}\right. 
	\end{equation}
	As a direct consequence,
	\[
	M_{\psi}(\rho_\phi)=M_{\psi}(\rho_{\phi,0})+M_\psi(\rho_{\phi,\geq1})\,,
	\]
	where we denote $\rho_{\phi,i}:=\int f_{\phi,i}\rd{v}\,,\quad\text{for}\quad i\in \{0,\geq 1\}$.
	
	Intuitively, the division of $f_\phi$ into the two components is to separate the particles that behave differently. In particular, $f_{\phi,0}$ denotes the number of bacteria on the phase space that tumble out of the state they were in. So in some sense, the distribution function ``decays" along the trajectory with rate $\sigma$. $f_{\phi,\geq 1}$, on the other hand, collects the distribution of all remaining bacteria. The right hand side of  equation \eqref{eqn:fgeq1} has three terms, representing the bacteria tumbling out of the state (thereby decaying in the distribution sense by $\sigma$), tumbling in from the source $f_{\phi,0}$ and tumbling in by $f_{\phi,\geq1}$. Since $f_{\phi,0}$ contains $\sigma$ information solely, one would expect to reconstruct $\sigma$ if $f_{\phi,0}$ information can be identified from the full $f_\phi$.
	
	The core of analysis lies in designing a special set of $\phi$ and $\psi$ that have compatible singularities to each other so that
	\begin{equation}\label{eqn:estimate_goal}
		M_\psi(\rho_\phi) = M_\psi(\rho_{\phi,0})\,,\quad\text{and}\quad M_\psi(\rho_{\phi,\geq1})=0\,,
	\end{equation}
	so we have access to the value of $M_\psi(\rho_{\phi,0})\approx M_\psi(\rho_\phi)$ that will be further used to reconstruct $\sigma$. When the context is clear below, we drop the $\phi$ dependence in $\rho$ to have a concise notation.
	

	We now list the conditions for $\phi$ and $\psi$ so to have~\eqref{eqn:estimate_goal} holds true. Let $\phi_x, \psi_x\in C_c^{\infty}(\rr^3)$, $\phi_v\in C_c^{\infty}(\rr^2)$ and $\psi_t\in C_c^{\infty}(\rr)$ be non negative functions  compactly supported in the unit ball $B^n(0,1)\subset\rr^n$
	\begin{align}\label{eq:phipsiproperties}
		&\supp(\phi_x), \supp(\psi_x)\subset B^3(0,1), \quad \supp(\phi_v)\subset B^2(0,1), \quad \supp(\psi_t)\subset B^1(0,1), \nonumber\\
		&0\leq \phi_x, \psi_x,\phi_v,\psi_t\leq 1\quad \text{with } \quad \phi_x(0)= \psi_x(0)=\phi_v(0)=\psi_t(0)=1\quad\text{and}\\
		&1=\int_{\rr^3}\phi_x(x)\, \rd x = \int_{\rr^3}\psi_x(x)\, \rd x = \int_{\rr^2}\phi_v(y)\, \rd y = \int_{\rr}\psi_t(t)\, \rd t.\nonumber
	\end{align}

	Let $(x_i,v_i)\in \rr^3\times V$ be the initial location and velocity of the bacteria concentration, and $(x_m,t_m)\in\rr^3\times (0,T)$ be the measurement location and time, we now set the initial data $\phi$ and measurement test function $\psi$ to be 
	\begin{alignat}{3}\label{eqn:def_phi_psi}
		\phi(x,v) &= 
		\frac{1}{\eps^3\delta^2}\phi_x\left(\frac{x-x_i}{\eps}\right)\phi_v\left(\frac{\pp_{v_i}(v)}{\delta}\right)j(v;v_i),
	&\quad\in C^{\infty}_c, 
	\nonumber\\
	\psi(x,t)&= \frac{1}{\eta}\psi_x\left(\frac{x-x_m}{\eps}\right)\psi_t\left(\frac{t-t_m}{\eta}\right)&\quad\in C^{\infty}_c
\end{alignat}
for small scaling parameters $\eps,\delta, \eta>0$. Furthermore, $\pp_{v_i}:\ss\setminus\{-v_i\} \to \rr^2$ denotes the stereographic projection on the direction of $-v_i$, with its absolute Jacobi determinant given by  $j(v;v_i) := 1/((1+\langle v,v_i\rangle)^2|\langle v,v_i\rangle|)$. Accordingly, we also define a quantity that will be used in the later discussion:
\begin{equation}\label{eqn:C_phipsi}
	C_{\phi,\psi}=\int_{\rr^3} \phi_x\left({x}\right)\psi_x\left({x}\right)\,\rd{x}\,.
\end{equation}
For the measurement $M_\psi(\rho_{\phi,0})$ to be non-trivial, the two pairs $(x_i,v_i)$ and $(x_m,t_m)$ should be compatible to each other. Indeed, we require
\[
x_m:= x_m(t_m)= x_i+v_it_m\,,
\]
so that the measurement location at $t_m$ indeed receives the data transported from $x_i$ in the direction of $v_i$.

The proof of the theorem is based on the following two lemmas.
\begin{lemma}\label{lem:f1}
	Let $\phi$ and $\psi$ be defined as in~\eqref{eqn:def_phi_psi}. Let $\sigma$ and $K$ be selected from the admissible sets. The solution to~\eqref{eqn:f0} gives
	\begin{align*}
		\lim_{\eps\to0}\lim_{\eta,\delta\to0}M_\psi(\rho_{0})= e^{-\int_0^{t_m} \sigma(x_i+v_is,v_i)\,  \rd{s}}  C_{\phi_x\psi_x}\,.
	\end{align*}
\end{lemma}
Similarly,
\begin{lemma}\label{lem:fgeq1}
	Let $\phi$ and $\psi$ be defined as in~\eqref{eqn:def_phi_psi}, with $t_m<T$ that satisfies $C_K|V|T<1$. Let $\sigma$ and $K$ be selected from the admissible sets. The solution to~\eqref{eqn:fgeq1} gives
	\begin{displaymath}
		\lim_{\eps\to 0}\lim_{\delta,\eta\to 0}M_\psi(\rho_{\geq 1}) =0\,.
	\end{displaymath}
\end{lemma}
Theorem~\ref{thm:unique_sigma} is a quick corollary of these two lemmas.
\begin{proof}[Proof of Theorem~\ref{thm:unique_sigma}]
	Under the conditions listed in Lemma~\ref{lem:f1} and Lemma~\ref{lem:fgeq1}, we have:
	\[
	\lim_{\eps\to0}\lim_{\eta,\delta\to0}M_\psi(\rho_\phi) = \lim_{\eps\to0}\lim_{\eta,\delta\to0}M_\psi(\rho_{0})=C_{\phi_x\psi_x}e^{-\int_0^{t_m} \sigma(x_i+v_is,v_i)\,  \rd s}\,.
	\]
	Then we have the immediate conclusion that:
	\begin{equation}\label{eq:recSigma}
		\sigma(x_m,v_i) = -\partial_{t_m}\ln\left(\frac{1 }{ C_{\phi_x\psi_x}}\lim_{\eps\to 0}\lim_{\delta, \eta\to 0} (M_{\psi}(\rho_\phi))\right).
	\end{equation}
\end{proof}
\begin{remark}
	The result above provides an explicit reconstruction of $\sigma$  in equation~\eqref{eq:recSigma}, however,  a stability bound is not given.  As the formula includes a time derivative outside the limit taking of the small parameters, additional difficulty may be  introduced in the stability analysis. Moreover, we would like to point out that neither  the proof of the theorem nor both lemmas show a specific dimension dependence which indicates that the result holds true for dimenstion $d=2$ as well. This is in contrast to classical kinetic inverse problem where an albedo operator is used to infer photon scattering coefficient. In these problems the data is typically confined on the boundary, i.e.  one dimension of freedom is lost, which explains the sensitivity of these results to the dimensionality of the problem. In our setup, data is taken in the interior and we have the additional freedom to adjust time. This appears to be sufficient to obtain independence of the dimension $d$, so the result easily extends.
\end{remark}

We now give proofs for the two lemmas above. It amounts to detailed calculations.

\begin{proof}[Proof of Lemma~\ref{lem:f1}]
	According to the equation for $f_0$ in~\eqref{eqn:f0}, we can explicitly compute $f_0$, along the trajectory of the bacteria propagation:
	\begin{equation}\label{eq:f0}
		f_0(x,t,v) = e^{-\int_0^t\sigma(x-vs,v)\,  \rd s} \phi(x-vt,v)\,.
	\end{equation}
	Inserting this into the definition of the measurement~\eqref{eqn:measurement}, we have
	\begin{align*}
		M_{\psi}(\rho_0) &= \int_0^T\int_{\rr^3}\int_V f_0(x,t,v)\,\rd v\,\psi(x,v)\,\rd x\,\rd t \\
		&= \int_0^T\int_{\rr^3}\int_V e^{-\int_0^t\sigma(x-vs,v)\,  \rd s} \phi(x-vt,v)\,\rd v\,\psi(x,t)\,\rd x\,\rd t\,.
	\end{align*}
	Plug in the form of $\phi$ and $\psi$ in~\eqref{eqn:def_phi_psi}, we have:
	\begin{align*}
		&M_{\psi}(\rho_0)\\
		=& \frac{1}{\eps^3\delta^2\eta} \int_0^T\int_{\rr^3}\int_V  e^{-\int_0^t \sigma(x-vs,v)\,  \rd s}\phi_x\left(\frac{x-vt-x_i}{\eps}\right)\phi_v\left(\frac{\pp_{v_i}(v)}{\delta}\right)j(v)\,\rd v\,\\
		&\hspace{3cm}\psi_x\left(\frac{x-x_m}{\eps}\right)\psi_t\left(\frac{t-t_m}{\eta}\right)\,\rd x\,\rd t \\
		=& \frac{1}{\eps^3} \int_{-\frac{t_m}{\eta}}^{\frac{T-t_m}{\eta}}\int_{\rr^3}\int_{\rr^2}  e^{-\int_0^{t_m +\eta\tilde{t}} \sigma(x-\pp_{v_i}^{-1}(\delta y)s,\pp_{v_i}^{-1}(\delta y))\,  \rd s}\\
		&\hspace{2cm}\phi_x\left(\frac{x-\pp_{v_i}^{-1}(\delta y)(t_m +\eta\tilde{t})-x_i}{\eps}\right)\phi_v\left(y\right)\psi_x\left(\frac{x-x_m}{\eps}\right)\psi_t\left(\tilde{t}\right)\,\rd y\,\rd x\,\rd\tilde{t}
	\end{align*}
	where we substituted $\tilde{t}:= \frac{t-t_m}{\eta}$ and $y:= \frac{\pp_{v_i}(v)}{\delta}$ in the last equation. Now fixing $\eps$ and letting $\eta\to0$, $\delta\to0$, then  by  continuity of $\sigma$,$\phi_x$, $\pp_{v_i}^{-1}$ and \eqref{eq:phipsiproperties}, the dominated convergence theorem yields
	\begin{align*}
		&\lim_{\eta,\delta\to0}M_\psi(\rho_{0})\\
		=&\frac{1}{\eps^3} \int_{\rr^3} e^{-\int_0^{t_m } \sigma(x-v_is,v_i)\,  \rd s}\phi_x\left(\frac{x-v_it_m-x_i}{\eps}\right)\psi_x\left(\frac{x-x_m}{\eps}\right)\,\rd x\\
		&\hspace{3cm}\cdot\int_{\rr}\psi_t\left(\tilde{t}\right)\,\rd\tilde{t}\int_{\rr^2} \phi_v\left(y\right)\,\rd y\\
		=& \int_{\rr^3} e^{-\int_0^{t_m} \sigma(x_m+\eps\tilde{x}-v_is,v_i)\,  \rd s}\phi_x\left(\tilde{x}\right)\psi_x\left(\tilde{x}\right)\,\rd\tilde{x}\,.
	\end{align*}
	We used the substitution $\tilde{x}:= \frac{x-x_m}{\eps}$ in the last equation. Now set $\epsilon\to0$ and use the continuity of $\sigma$, we obtain
	\begin{align*}
		&\lim_{\eps\to0}\lim_{\eta,\delta\to0}M_\psi(f_{\phi,0})\\
		=&e^{-\int_0^{t_m} \sigma(x_m-v_is,v_i)\,  \rds}\int_{\rr^3} \phi_x\left(\tilde{x}\right)\psi_x\left(\tilde{x}\right)\,\rd\tilde{x} = e^{-\int_0^{t_m} \sigma(x_i+v_is,v_i)\,  \rds}  C_{\phi_x\psi_x}\,,
	\end{align*}
	where we used~\eqref{eqn:C_phipsi}. The proof is concluded.
\end{proof}

To prove Lemma \ref{lem:fgeq1}, we will first introduce the following Lemma. 
\begin{lemma}\label{lem:estMfgeqN}
	Let $g$ satisfy the following equation:
	\begin{equation}\label{eq:g}
		\begin{cases}
			\partial_t g +v\cdot\nabla g = -\sigma g +\mathcal{L}(g) +\mathcal{L}(h)\,,&(x,t,v)\in \rr^3\times [0,T]\times V\\
			g(x,t=0,v) = 0\,,
		\end{cases}
	\end{equation}
	where $\mathcal{L}$ and $\sigma$ are defined in~\eqref{eqn:def_L}-\eqref{eqn:def_sigma} for $K\in\mathcal{A}_K$ and $h$ is a given positive function, then the measurement of $g$ with respect to a general  measurement test function $\psi\in C_c^{\infty}$ is bounded by 
	\begin{align}\label{eq:recKMfgeq2=OMf2}
		M_\psi(\<g\>) \leq& C_K|V|e^{C_K|V|T}M_\psi\left( \int_0^t\esup_x(\<h\>(x,s))\, \rd s\right) \,.
	\end{align}
\end{lemma}
\begin{proof}
	The proof is a direct calculation. Integrate~\eqref{eq:g} along the characteristics, we have
	\begin{align*}
		&\esup_x\<g\>=\esup_x\int_V g (x,t,v)\, \rdv \\
		&= \esup_x\int_V\int_0^t \left[-\sigma g +\mathcal{L}(g) + \mathcal{L}(h) \right](x-vs,t-s,v)\, \rds\, \rdv \\
		&\leq C_K \int_V\int_0^{t} (\esup_x\<g\>+\esup_x\<h\>) (x-vs,t-s)) \, \rds\, \rdv \\
		&= C_K|V|\int_0^{t}\esup_x(\<g\>  (x,s)) \, \rds +\underbrace{C_K|V|\int_0^{t} \esup_x(\<h\> (x,s)) \, \rds}_{=:\alpha(t)}
	\end{align*}
	where we used the positivity of $g$~\cite{Majorana_Positivityf_1997} and $\sigma$ as well as the boundedness of $K$ in the inequality and a change of variables. 
	We call the integral form of Gronwall's lemma 
	and use the fact that $g(x,t=0,v)=0$ and $\alpha$ is non-decreasing in order to obtain:
	\begin{align*}
		\esup_x\<g\>\leq & \alpha(t)e^{\int_0^tC_K|V|\rds} = C_K|V|e^{C_K|V|t}\int_0^{t} \esup_x(\<h\> (x,s)) \, \rds
	\end{align*}
	Noting that $M_\psi(\<g\>)$ 
	is a linear operator and preserves the monotonicity, we conclude~\eqref{eq:recKMfgeq2=OMf2}.
\end{proof}
With the above Lemma at hand, we can readily prove Lemma \ref{lem:fgeq1}.

\begin{proof}[Proof of Lemma \ref{lem:fgeq1}]\label{proof:lemmaMfgeq1To0}
	We further decompose $f_{\geq 1} = f_1+f_2+...+f_{N}+f_{\geq N+1}$ for some  $N\in \mathbb{N}$ to be chosen later, with each level of $f_n$, $n\geq 1$, satisfying
	\begin{align}\label{eq:fneq}
		\begin{cases} \partial_t f_n +v\cdot \nabla f_n = -\sigma f_n + \mathcal{L}{f_{n-1}}\,,\\
			f_n(x,t=0,v) = 0\,,
		\end{cases}
	\end{align}
	and the last level
	\begin{align*}
		\begin{cases} \partial_t f_{\geq N+1} +v\cdot \nabla f_{\geq N+1} = -\sigma f_{\geq N+1} + \mathcal{L}{f_{ N}}+ \mathcal{L}{f_{\geq N+1}} \,,\\
			f_{\geq N+1}(x,t=0,v) = 0\,.
		\end{cases}
	\end{align*}
	Then the measurement decomposes accordingly, i.e.
	\[
	M(\rho_{\geq 1}) = M(\rho_{ 1}) + M(\rho_{2})+...+M(\rho_{N}) + M(\rho_{\geq N+1})\,.
	\]
	Our objective is to show that in the scaling limit, all $M(\rho_i)$ vanish for $i\leq N$ and $M(\rho_{\geq N+1})$ is arbitrarily small for a big $N$.
	\begin{itemize}[leftmargin=*]
		\item 
		To do so we first write an explicit expression for $M(\rho_n)$ for an arbitrary $n\in \mathbb{N}$. We integrate \eqref{eq:fneq} along characteristics and use the fact that $\sigma, f_n$ are nonnegative to see
		\begin{align*}
			f_n(x,t,v_0)& = \int_0^t [-\sigma f_n + \mathcal{L}(f_{n-1}) ](x-v_0s_0,t-s_0,v_0)\, \rd s_0 \\
			&\leq C_K  \int_0^t  \<f_{n-1}\>(x-v_0s_0,t-s_0)\, \rd s_0 \\
			&= C_K  \int_0^t \int_V f_{n-1}(x-v_0s_0,t-s_0,v_1)\, \rd v_1\,\rd s_0\,.
		\end{align*}
		In this notation, $v_0$ is the last direction in which the bacteria of $f_n$ run, and $s_0$ the time  for which they run into this direction. Respectively, $v_j,$ and $s_j$ denote the direction and time in which the bacteria run after their $(n-j)$-th tumble.
		
		By induction,
		\begin{align}\label{eq:estf_n}
			f_n(x,t,v_0) \leq C_K^{n} \int_0^t \int_V \int_0^{t-s_0}\int_V...\int_0^{t-\sum\limits_{j=0}^{n-2}s_j}\int_V f_0\left(x-\sum_{j=0}^{n-1} s_jv_j, t-\sum_{j=0}^{n-1} s_j, v_n \right)\\
			\, \rd v_n \, \rd s_{n-1}...\rd v_2\, \rds_1\, \rdv_1\, \rds_0\nonumber\\
			\leq C_K^{n} \int_0^t \int_V ...\int_0^{t-\sum\limits_{j=0}^{n-2}s_j}\int_V \phi\left(x-\sum_{j=0}^{n-1} s_jv_j-\left( t-\sum_{j=0}^{n-1} s_j\right) v_n, v_n \right)\nonumber\\
			\, \rd v_n \, \rd s_{n-1}... \rdv_1\, \rds_0,\nonumber
		\end{align}
		where we bounded $f_0(x,t,v)$ by $\phi(x-vt,v)$ using~\eqref{eq:f0} and noting $\sigma\geq 0$. Inserting this into the measurement and calling the dominated convergence theorem, we have
		\begin{align*}
			&M(\rho_n) = \int_{\rr^3}\int_0^T \int_V f_n(x,t,v_0)\, \rdv_0\, \psi(x,t) \, \rdt\, \rdx \\
			&\leq  C_K^{n}\int_{\rr^3}\int_0^T \int_V \bigg[\int_0^t \int_V ...\int_0^{t-\sum\limits_{j=0}^{n-2}s_j}\int_V \phi\left(x-\sum_{j=0}^{n-1} s_jv_j-\left( t-\sum_{j=0}^{n-1} s_j\right) v_n, v_n \right)\\
			&\hspace{6cm}\, \rd v_n \, \rd s_{n-1}... \rdv_1\, \rds_0\,\bigg]\,\rdv_0 \,\psi(x,t)\,  \rd t \, \rd x\\
			&\xrightarrow[]{\delta, \eta\to 0}  C_K^{n}\int_{\rr^3} \int_V \int_0^{t_m} \int_V ...\int_0^{t_m-\sum\limits_{j=0}^{n-2}s_j}\frac{1}{\eps^3}\psi_x\left(\frac{x-x_m}{\eps}\right)\cdot\\
			&\hspace{1.5cm}\phi_x\left(\frac{x-\sum_{j=0}^{n-1} s_jv_j-\left( t_m-\sum_{j=0}^{n-1} s_j\right) v_i-x_i}{\eps}\right)
			\, \rd s_{n-1}... \rdv_1\, \rds_0\,\rdv_0 \,\,  \rd x\\
			&= C_K^{n}\int_{\rr^3} \int_V \int_0^{t_m} \int_V ...\int_0^{t_m-\sum\limits_{j=0}^{n-2}s_j}\psi_x\left(\tilde{x}\right)\cdot\\
			&\hspace{4cm}\phi_x\left(\tilde{x} +\frac{\sum_{j=0}^{n-1} s_j(v_i-v_j)}{\eps}\right)
			\,\rd s_{n-1}... \rdv_1\, \rds_0\,\rdv_0 \,  \rd \tilde{x}.
		\end{align*}
		In the last line we used the substitution $\tilde{x}=\frac{x-x_m}{\eps}=\frac{x-x_i- v_it_m}{\eps}$. Now for $\eps\to 0$, one has
		\begin{align*}
			\phi_x\left(\tilde{x}+\frac{\sum_{j=0}^{n-1} s_j(v_i-v_j)}{\eps}\right)\to
		\phi(\tilde{x})\mathds{1}_0\left(\sum_{j=0}^{n-1} s_j(v_i-v_j))\right),
	\end{align*}
	where $\mathds{1}$ denotes the indicator function\footnote{$\mathds{1}_A(a) = 1$ for $a$ is in the set $A$ and zero else and $\mathds{1}_{a'} := \mathds{1}_{\{a'\}} $ for elements $a'$ in some set.}. Using the dominated convergence theorem again, we have:
	\begin{align*}
		\lim_{\eps\to 0}\lim_{\delta,\eta\to 0}M(\rho_n)&\leq C_{\phi,\psi} C_K^{n} \int_V \int_0^{t_m} \int_V ...\int_0^{t_m-\sum_{j=0}^{n-2}s_j}\mathds{1}_{0}\left(\sum_{j=0}^{n-1} s_j(v_i-v_j)\right)\\
		&\hspace{6cm}
		\,\rd s_{n-1}... \rdv_1\, \rds_0\,\rdv_0\\
		&= 0,
	\end{align*}
	where the last equality holds true, because the integration is taken on a measure-zero set for a bounded integrand.
	We conclude $M_\psi(\rho_{n})=0$ in the limit of $\epsilon,\delta,\eta\to0$.
	\item
	We now proceed to show the smallness of $M(\rho_{\geq N+1})$. Apply Lemma~\ref{lem:estMfgeqN} with $g:= f_{\geq N+1}, h:= f_{N}$, we have
	\begin{align}\label{eq:MfgeqNbound}
		M(\rho_{\geq N+1}) \leq &C_K |V|e^{C_K|V|T}M\left(\int_0^t \esup_x\<f_N\>(x,s)\,\rds\right).
	\end{align}
	Using estimate \eqref{eq:estf_n} for $f_N$ as well as   $\max_x\phi(x,v) \leq \frac{1}{\eps^3\delta^2}\phi_v\left(\frac{\pp_{v_i}(v)}{\delta}\right)j(v;v_i)$, we have
	\begin{align*}
		&\int_0^t \esup_x\<f_N\>(x,s)\,\rds\\
		&\leq C_K^ {N} \int_0^t\int_V\bigg[\int_0^{s}\int_V...\int_0^{s-\sum\limits_{j=0}^{N-2}s_j}\int_V\frac{1}{\eps^3\delta^2} \phi_v\left(\frac{\pp_{v_i}(v_N)}{\delta}\right)j(v_N;v_i)\\
		&\hspace{7cm}\rdv_N\,\rds_{N-1}...\rdv_1\, \rds_0\bigg]\,\rdv_0\,\rds\\
		&\leq \frac{1}{\eps^3}C_K^ {N}|V|^{N}t^{N+1}
		\int_V\frac{1}{\delta^2} \phi_v\left(\frac{\pp_{v_i}(v_N)}{\delta}\right)j(v_N;v_i)\,\rdv_N.
	\end{align*}
	Since the above integral over $v_N$ has value $1$,  \eqref{eq:MfgeqNbound} gives
	\begin{align*}
		M(\rho_{\geq N+1})
		&    \leq (C_K|V|T)^{N+1} e^{C_K|V|T}
		\int_{\rr^3}\int_0^T\frac{1}{\eps^3}\psi(x,t)\, \rdt\, \rdx \\
		&\leq (C_K|V|T)^{N+1} e^{C_K|V|T}.
	\end{align*}
	This shows $M(\rho_{\geq N+1}) $ becomes arbitrarily small as $N$ grows, when  $C_K|V|T<1$.  
	In summary, this proves Lemma \ref{lem:fgeq1}.
\end{itemize}
\end{proof}

\section{Reconstructing $K$}\label{sec:RecK}
This section is dedicated to the reconstruction of $K$ from macroscopic measurements of the bacteria density $\rho$. The idea is similar to the previous section: A class of special functions are used as the initial conditions, and accordingly the measurement test functions are designed. These functions carry certain type of singularity and are designed to be compatible with each other, so the measurement singles out a trajectory that we would like to get information about. In the end, we will prove the following theorem.
\begin{theorem}[Unique reconstruction of $K$]\label{thm:recK}
Let $\sigma\in\mathcal{A}_{\sigma}$ and $K\in\mathcal{A}_{K}$. The map $\Lambda_K$ uniquely determines $K(x,v,v')$. In particular, for any $(x,v,v')$, by a proper choice of $\phi$ and $\psi$, one can explicitly express $K(x,v,v')$ in terms of $M_\psi(\rho_\phi)$, with $\rho_\phi$ being the density associated with $f_\phi$ that solves~\eqref{eq:chemotaxis}.

\end{theorem}
Note 
that $K$ is the tumbling kernel, so the reconstruction necessarily needs at least one scatter. To do so, we decompose $f$ into three, instead of two parts. Let $f_\phi= f_{\phi,0}+f_{\phi,1}+f_{\phi,\geq 2}$, where $f_{\phi,0}$ solves~\eqref{eqn:f0} using $\phi$ as the initial data, and $f_{\phi,1}$ and $f_{\phi,\geq 2}$ solve the following:
\begin{align} \label{eq:f_1PDE}
&\left\{\begin{array}{r l}\partial_t f_{\phi,1}+v\cdot\nabla f_{\phi,1} &= -\sigma f_{\phi,1} +\L(f_{\phi,0}),\\
	f_{\phi,1}(x,t=0,v)& = 0,\end{array}\right.,
\end{align}
and
\begin{align}\label{eq:fgeq2PDE}
&\left\{\begin{array}{r l}\partial_t f_{\phi,\geq 2} +v\cdot \nabla f_{\phi,\geq 2} &= -\sigma f_{\phi,\geq 2} +\L(f_{\phi,1}+f_{\phi,\geq 2}),\\
	f_{\phi,\geq 2}(x,t=0,v) &= 0.\end{array}\right.
\end{align}
To reconstruct $K$, we will design a special set of test functions and initial conditions so to have, in certain scenarios:
\begin{equation}\label{eqn:estimate_goal2}
M_{\psi}(\rho_\phi) = M_{\psi}(\rho_{\phi,1}),\quad  M_{\psi}(\rho_{\phi,0})= 0 = M_{\psi}(\rho_{\phi,\geq 2})\,,
\end{equation}
and the measurement $ M_{\psi}(\rho_{\phi,1})$ is expected to give sufficient information to reconstruct $K$. As in the previous section, we omit the $\phi$ dependence in $\rho$ and $f$ when the context is clear.

As in the previous case, this will again hold in the limit as the initial data and measurement test functions become singular functions concentrated on initial velocity $v_i$ and  location  $x_i$ and measurement location $x_m$ and time $t_m$ respectively. 

Unlike in the previous case, we require $x_m$ to avoid the line formed by $x_i+v_it_m$ so to ensure at least one scatter. This means we require the particle to initially travel with velocity $v_i$ and change its direction to another $\hat{v}$ at a certain time $t_m-\hat{s}$.  See the illustration in Figure~\ref{fig:ellipse}. Our measurement location is thus be chosen as
\begin{equation}\label{eqn:tumble_point_constraints}
x_m =x_m(t_m)= x_i +\hat{s} \hat{v}+  (t_m-\hat{s})v_i \,,\quad\text{with}\quad \hat{s}=\lambda t_m, \lambda\in (0,1)\,, \hat{v}\in V\setminus\{v_i\}.
\end{equation}
Let us mention that for the fixed triplet of $(x_i,x_m,t_m)$, the quantities  $\hat{s},\hat{v}$ and thus  tumbling point at $(x_i+v_i(t_m-\hat{s}),\hat{v},v_i)$, that contributes information to the measurement, are uniquely determined, see Figure~\ref{fig:ellipse}.

Moreover, by gradually shrinking $t_m$ and $x_m-x_i$ in a suitable manner we  make  the ellipse of observation shrink while   the geometry features are sustained, in particular, $\hat{s}$ will stay a fixed fraction of $t_m$, pushing the tumbling point close to the starting point.

Because our results will hold asymptotically for small $t_m$, we wish to sustain the geometry as $t_m\to 0$ and choose $\hat{s}$  as a fixed fraction of $t_m$. We note that given such $x_i,x_m,t_m$, then $\hat{s},\hat{v}$ are unique, compare Figure~\ref{fig:ellipse}.


\begin{figure}[H]
\centering
\includegraphics[height = 3.5cm]{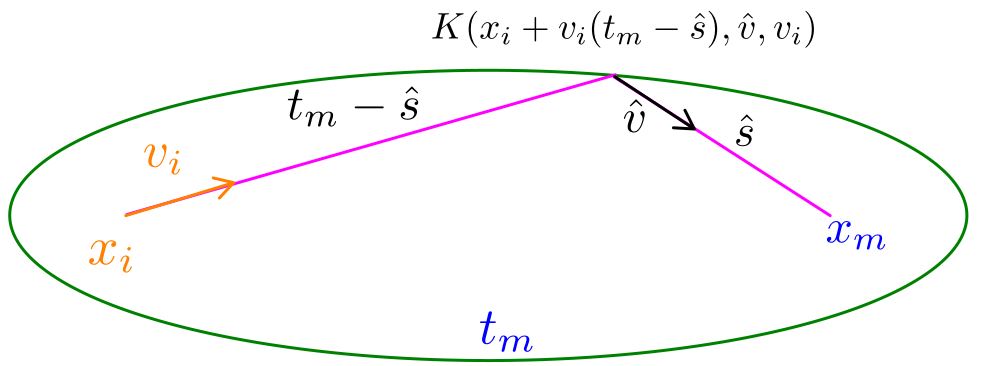}
\caption{For a fixed  $t_m>0$, the ellipse with focal points $x_i,x_m$ and radius $t_m$ determines all points $x$ with distance $\|x-x_i\|+\|x-x_m\|=t_m$. As $v_i$ is given, the unique tumbling point $x_i+v_i(t_m-\hat{s})$ is the intersection of the half line starting at $x_i$ in direction $v_i$ with this ellipse.}
\label{fig:ellipse}
\end{figure}

More specifically, let $\phi_x,\phi_v,\psi_x,\psi_t$ be defined as in \eqref{eq:phipsiproperties}, and we set $t_m:= \eps^{\alpha}$ for some $\alpha\in \left({\frac{3}{4}},1\right)$, then we let
\begin{alignat}{3}
\phi(x,v) &= 
\frac{1}{\eps^3\delta^2}\phi_x\left(\frac{x-x_i}{\eps}\right)\phi_v\left(\frac{\pp_{v_i}(v)}{\delta}\right)j(v;v_i), 
&\quad\in C^{\infty}_c 
\label{eq:initdataK}\\
\psi(x,t)&= \frac{1}{\nu^3\eta}\psi_x\left(\frac{x-x_m}{\nu}\right)\psi_t\left(\frac{t-t_m}{\eta}\right)C_{\hat{s},\hat{v}}&\quad\in C^{\infty}_c\label{eq:testfK}
\end{alignat}
for small scaling parameters $\eps,\delta,\eta,\nu>0$ and the constant 
\begin{equation}\label{eq:Cshatvhat}
C_{\hat{s},\hat{v}}:= \hat{s}^2 (1-\< v_i,\hat{v}\>)\,.
\end{equation}
Note that the scaling of   $\psi$ is different from the one in~\eqref{eqn:def_phi_psi}, in particular the scalings of $x$ are different in $\phi$ and $\psi$. We request $t_m=\eps^{\alpha}\xrightarrow{\eps\to 0} 0$. Note that $\epsilon$ is the rate at which $\phi$ is converging to a delta-measure in $x$. Since  $\alpha<1$, the convergence of the observation time $t_m\to 0$ is slightly slower than the convergence of the initial condition $\phi$.
The observation of  the propagation of singularities typically relies on small time requirements, see e.g.~\cite{Doumic_Tournus_ReconstructFragmentationKernelShortTime}. In this paper, we additionally exploit a particular relation between time and spatial scaling that appears to be beneficial to  control  the multiple tumble part $M_\psi(\rho_{\phi,\geq 2})$.

We claim for this setup, with $t_m$ being very small, we will be able to achieve the estimate~\eqref{eqn:estimate_goal2}, and the measurement exactly reflects the value of $K(x_i+v_i(t_m-\hat{s}),\hat{v},v_i)$, as seen in the following Lemmas.
\begin{lemma}\label{lem:recK:Mf_0to0}
Let $\sigma$ be from the admissible set, and let the $(\phi,\psi)$ pairs  be defined as in~\eqref{eq:initdataK}-\eqref{eq:testfK}, then $M_{\psi}(\rho_{0})$ vanishes in the limit, meaning:
\begin{equation}\label{eq:recKMf0limit0}
	\lim_{\eps\to 0}\lim_{\delta, \nu,\eta \to 0} M_{\psi}(\rho_{0})=0\,.
\end{equation}
\end{lemma}
\begin{lemma}\label{lem:recK:f1}
Let $K$ be from the admissible set, and let the initial data and test functions defined in \eqref{eq:initdataK}-\eqref{eq:testfK}, then the measurement  $M_{\psi}(\rho_{1})$ reconstruct $K$, in the sense that
\begin{equation}\label{eq:f1LimitNoTime}
	\lim_{\eps\to 0}\lim_{\delta, \nu,\eta \to 0}M_{\psi}(\rho_1)=  K(x_i,\hat{v},v_i)\,,
\end{equation}
where $\hat{v}$ is the velocity after tumbling used in the construction of the measurement location $x_m$ in~\eqref{eqn:tumble_point_constraints}. 
\end{lemma}
\begin{lemma}\label{lem:recK:fgeq2}
Assume $K$ and $\sigma$ are bounded and positive, then
\begin{equation}\label{eq:recK:Mf_geq2to0}
	\lim_{\eps\to 0}\lim_{\delta, \nu,\eta \to 0}M_{\psi}(\rho_{\geq 2})=  0.
\end{equation}
\end{lemma}
Together, these lemmas prove Theorem \ref{thm:recK}.


\begin{proof}[Proof of Theorem \ref{thm:recK}]
Combining the previous three lemmas, we have in the limit, equation~\eqref{eqn:estimate_goal2} holds true, meaning:
\begin{equation}\label{eq:recK:limits}
	\lim_{\eps\to 0}\lim_{\delta,\nu,\eta\to 0} M_\psi(\rho_f^\phi) = \lim_{\eps\to 0}\lim_{\delta,\nu,\eta\to 0} M_\psi(\rho_1) =K(x_i,\hat{v},v_i).
\end{equation}

\end{proof}

The rest of this section is dedicated to showing Lemma~\ref{lem:recK:Mf_0to0},~\ref{lem:recK:f1} and~\ref{lem:recK:fgeq2}.

Lemma~\ref{lem:recK:Mf_0to0} states that the contribution from $f_0$ in the measurement vanishes as initial and measurement functions become singular. This is intuitively straightforward. Indeed, as illustrated in Figure~\ref{fig:f0awayxm}, $x_m$ is not along the straight line from $x_i$ in the direction of $v_i$, so if $\phi$ is singular enough, $x_m$ lies out of the support of $f_0$.
\begin{figure}[H]
\centering
\includegraphics[width = .5\textwidth]{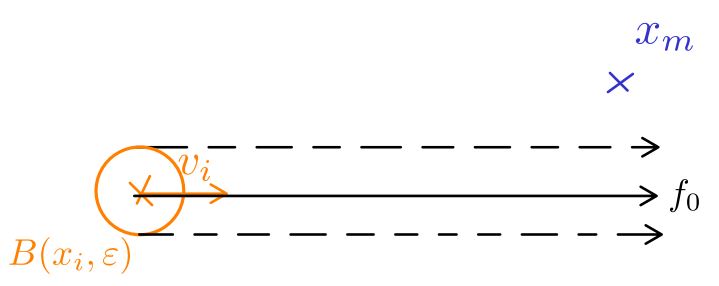}
\caption{Considering the situation of a point measurement of $f_0$ at  $(x_m,t_m)$ where the  initial velocity is prescribed $\phi(x,v) = \tilde{\phi}(x)\delta_{v_i}(v)$, the support of $f_0(\cdot, t_m,v_i)$ equals $\supp\{\tilde{\phi}(x_i +v_i t_m)\}\subset B(x_i+v_i t_m,\eps)$, the translated support of $\tilde{\phi}$. When $\eps$ becomes small, at some point $B(x_i+v_i t_m,\eps)$ no longer contains $x_m$, since  $x_m\neq x_i+v_i t_m$.}
\label{fig:f0awayxm}
\end{figure}

\begin{proof}[Proof of Lemma \ref{lem:recK:Mf_0to0}]
Recalling the solution to~\eqref{eqn:f0} is in~\eqref{eq:f0}. If we test it with $\psi$, and make use of the change of variable: $\tilde{t}=\frac{t-t_m}{\eta}$,  $y=\frac{\pp_{v_i}(v)}{\delta}$ and $\tilde{x}=\frac{x-x_m}{\eps}$, we have:
\begin{align*}
	&M_{\psi}(\rho_0) = \int_0^T\int_{\rr^3}\int_V f_0(x,t,v)\,\rdv\,\psi(x,t)\,\rd x\,\rd t \\
	=& \frac{{C_{\hat{s},\hat{v}}}}{\eps^3\delta^2\nu^3\eta} \int_0^T\int_{\rr^3}\int_V  e^{-\int_0^t \sigma(x-vs,v)\,  \rd s}\phi_x\left(\frac{x-vt-x_i}{\eps}\right)\cdot\\
	&     \hspace {2cm} \phi_v\left(\frac{\pp_{v_i}(v)}{\delta}\right)j(v)\,dv\,\psi_x\left(\frac{x-x_m}{\nu}\right)\psi_t\left(\frac{t-t_m}{\eta}\right)\,\rd x\,\rd t\\
	= &\frac{{C_{\hat{s},\hat{v}}}}{\eps^3} \int_{-\frac{t_m}{\eta}}^{\frac{T-t_m}{\eta}}\int_{\rr^3}\int_{\rr^2}  e^{-\int_0^{t_m +\eta\tilde{t}} \sigma(x_m +\nu \tilde{x}-\pp_{v_i}^{-1}(\delta y)s,\pp_{v_i}^{-1}(\delta y))\,\rd s}\\
	&\hspace{1cm}\phi_x\left(\frac{x_m +\nu \tilde{x}-\pp_{v_i}^{-1}(\delta y)(t_m +\eta\tilde{t})-x_i}{\eps}\right)\cdot\phi_v\left(y\right)\psi_x\left(\tilde{x}\right)\psi_t\left(\tilde{t}\right)\,\rd y\,\rd\tilde{x}\,\rd\tilde{t}\\
	& \xrightarrow[]{\delta,\nu,\eta\to 0} \frac{{C_{\hat{s},\hat{v}}}}{\eps^3} e^{-\int_0^{t_m } \sigma(x_m -v_is,v_i)\,  \rd s}\phi_x\left(\frac{x_m -v_it_m -x_i}{\eps}\right)\,.
\end{align*}
In the last step, we exchange the limit with the integration using the dominated convergence theorem which is applicable because of the continuity of $\sigma$ and $\phi_x$. Noting the construction of $x_m$ in \eqref{eqn:tumble_point_constraints}, 
\begin{displaymath}
	\left\|\frac{x_m -v_it_m -x_i}{\eps} \right\|= \frac{\|\hat{s}(\hat{v}-v_i)\|}{\eps} = \|\lambda(\hat{v}-v_i)\|\eps^{\alpha-1}>1
\end{displaymath}
for  any small enough, but fixed, $\eps>0$, making $\phi_x\left(\frac{x_m -v_it_m -x_i}{\eps}\right) = 0$ according to the definition of $\phi_x$. This proves \eqref{eq:recKMf0limit0}.
\end{proof}

\begin{proof}[Proof of Lemma \ref{lem:recK:fgeq2}]
Repeating the arguments in the \hyperref[proof:lemmaMfgeq1To0]{ proof of Lemma \ref{lem:fgeq1}} , we estimate the remainder
\begin{align*}
	M(\rho_{\geq 2})& \leq \frac{1}{\eps^3} C_K^2|V|e^{C_K|V|T}M\left(\int_0^t s\, \rd s\right)=  \frac{1}{\eps^3} C\int_0^T\int_{\rr^3} t^2\,\psi(x,t)\, \rdx\,\rd t\, C_{\hat{s},\hat{v}}\\
	&=   \frac{C}{\eps^3} \tilde{C}\int_0^T t^2\frac{1}{\eta}\psi_t\left(\frac{t-t_m}{\eta}\right)\,\rd t\, \hat{s}^2
\end{align*}
where $C:=C_K^2|V|e^{C_K|V|T}/2$ and $\tilde{C}:= (1-\<\hat{v},v_i\>)$ are positive constants and we used  $f_1 \leq  t C_K\eps^{-3}$, which can be seen in \eqref{eq:estf_n}. We employ the dominated convergence theorem to the right hand side to see
\begin{align*}
	\lim_{\eta\to 0} M(\rho_{\geq 2})\leq \frac{C}{\eps^3} \tilde{C}t_m^2 \hat{s}^2 = C\tilde{C}\lambda^2 \eps^{4\alpha -3} \xrightarrow{\eps\to 0} 0,
\end{align*}
because $\alpha>\frac{3}{4}$ was chosen. Together with the non-negativity of $\rho_{\geq 2}$, we conclude \eqref{eq:recK:Mf_geq2to0}.
\end{proof}

Lemma~\ref{lem:recK:f1} lies at the core of Theorem~\ref{thm:recK}, and the proof largely depends on explicit derivation. Noting that according to~\eqref{eqn:tumble_point_constraints}, with fixed $(x_i,v_i)$ and $(x_m,t_m)$, one finds a unique local point for the bacteria to tumble: $x_i+v_i(t_m-\hat{s})$, so the measurement should reflect $K$ evaluated at this particular point.
\begin{proof}[Proof of Lemma \ref{lem:recK:f1}]
We first derive a closed form for $f_1$ by solving~\eqref{eq:f_1PDE} along characteristics:
\begin{align}\label{eq:f_1}
	f_1(x,t,v) =& \frac{1}{\eps^3\delta^2}\int_0^t\int_V e^{-\int_0^s \sigma(x-v\tau, v) \rd\tau}  K(x-vs,v,v') \cdot \nonumber\\
	&\hspace{1cm}e^{-\int_0^{t-s} \sigma(x-vs-v'\tau,v')\,  \rd\tau}\phi_x\left(\frac{x-vs-v'(t-s)-x_i}{\eps}\right)\cdot \\ &\hspace{1cm}\phi_v\left(\frac{\pp_{v_i}(v')}{\delta}\right)j(v';v_i)\,\rd v'\,\rd s\,,\nonumber
\end{align}
where we have used the explicit solution $f_0$ as in~\eqref{eq:f0}. Plugging it into the definition of the measurement:
\begin{equation*}
	\begin{aligned}
		M_{\psi}(\rho_1) =& \int_0^T\int_{\rr^3}\int_V f_1(x,t,v)\,\rd v\,\psi(x,t)\,\rd x\,\rd t \\
		=&  \int_0^T\int_{\rr^3}\int_V  \int_0^t\int_V \frac{{C_{\hat{s},\hat{v}}}}{\eps^3\delta^2\nu^3\eta}e^{-\int_0^s \sigma(x-v\tau, v) \rd\tau}  K(x-vs,v,v') \cdot\\
		&\hspace{1cm}e^{-\int_0^{t-s} \sigma(x-vs-v'\tau,v')\,  \rd\tau}\phi_x\left(\frac{x-vs-v'(t-s)-x_i}{\eps}\right)\cdot\\
		&\hspace{1cm}\phi_v\left(\frac{\pp_{v_i}(v')}{\delta}\right)j(v')\rd v'\,\rd s\,\rd v\,\psi_x\left(\frac{x-x_m}{\nu}\right)\psi_t\left(\frac{t- t_m}{\eta}\right)\,\rd x\,\rd t\,.
	\end{aligned}
\end{equation*}
Taking the limit and use dominant convergence theorem:
\begin{equation}\label{eqn:M_1_K}
	\begin{aligned}
		&\lim_{\delta,\nu,\eta\to 0}M_{\psi}(\rho_1)\\
		=& \int_0^{t_m}\int_V\frac{{C_{\hat{s},\hat{v}}}}{\eps^3} e^{-\int_0^s \sigma(x_m-v\tau, v) \rd\tau}  K(x_m-vs,v,v_i) \cdot\\
		&\hspace{0.3cm}e^{-\int_0^{t_m-s} \sigma(x_m-vs-v_i\tau,v_i)\,  \rd\tau}\phi_x\left(\frac{x_m-vs-v_i(t_m-s)-x_i}{\eps}\right)\,\rd v\,\rd s \,.
	\end{aligned}
\end{equation}

The formula above could be further reduced if noticing the support condition for $\phi_x$. In particular, we denote the argument of $\phi_x$, $\frac{x_m-vs-v_i(t_m-s)-x_i}{\eps}$ to be $\mathfrak{a}$,  it is straightforward to see that
$$
\|\mathfrak{a}\| \geq \frac{\|x_m-x_o\| - s\|v-v_i\|}{\eps}\,,
$$
where we denote
\begin{equation}\label{eqn:x_o}
	x_o = x_i+v_it_m\,,
\end{equation}
the location of the particle at time $t_m$ assuming it does not tumble. For small but fixed $\eps>0$, this further gives:
\begin{itemize}
	\item When $s<c_1:=\frac{\|x_m-x_o\|}{4}$, since $\|v-v_i\|\leq 2$
	\begin{align}\label{eqn:time_restriction}
		\|\mathfrak{a}\| \geq  \frac{\|x_m-x_o\| - 2s}{\eps} =  \frac{\|x_m-x_o\|}{2\eps} > 1\,.
	\end{align}
	\item When $\<v_i, v\>>1-c_2:=1-\frac{\|x_m-x_o\|^2}{8t_m^2}$, since $\|v-v_i\|=\sqrt{2-2\<v_i,v\>}$
	\begin{align}\label{eqn:angle_restriction}
		\|\mathfrak{a}\| > \frac{\|x_m-x_o\| - s\sqrt{2\frac{\|x_m-x_o\|^2}{8t_m^2}}}{\eps}\geq \frac{\|x_m-x_o\| }{2\eps}>1\,.
	\end{align}
\end{itemize}
That means that the integrand in~\eqref{eqn:M_1_K} would be $0$ due to the finite support of $\phi_x$ in these two parts of the domain.

In this reduced domain, $U=\{(s,v)\in[c_1,t_m]\times \{v\in V\mid \<v,v_i\>\leq 1-c_2\}\}$, we define the function $\mathcal{S}: (s,v)\mapsto z:= s(v-v_i)$. We note this function is injective. For fixed $z$ in its image, we can calculate its inverse:
\begin{equation}\label{eq:ZetaOmega}
	\mathcal{S}^{-1}(z) = (\zeta, \omega)(z) = \left(\frac{|z|^2}{2|\<z,v_i\>|}, v_i + \frac{z}{\zeta(z)}\right)\,.
\end{equation}
So we conduct change of variable by letting $z=\mathcal{S}(s,v)$, we further reduce the domain of~\eqref{eqn:M_1_K} to $\mathcal{S}(U)$ {and use the definition \eqref{eq:Cshatvhat} of $C_{\hat{s},\hat{v}}$}
\begin{align*}
	&\lim_{\delta,\nu,\eta\to 0}M_{\psi}(\rho_1)\\
	&=
	\frac{1}{\eps^3} \int_{\mathcal{S}(U)}  e^{-\int_0^{\zeta(z)} \sigma(x_m-\tau \omega(z), \omega(z)) \rd\tau} K(x_m-\zeta(z)\omega(z),\omega(z),v_i)\cdot\\ &\hspace{1cm} e^{-\int_0^{t_m-\zeta(z)} \sigma(x_m-\zeta(z)\omega(z)-v_i\tau,v_i)\,  \rd\tau}\phi_x\left(\frac{x_m-z-x_o}{\eps}\right)\cdot\\ &\hspace{1cm}\frac{\hat{s}^2}{\zeta(z)^2}\frac{1-\< v_i,\hat{v}\>}{1-\<v_i,\omega(z)\>}\,\rd z\\
	&=\int_{\frac{a-\mathcal{S}(U)}{\eps}} e^{-\int_0^{\zeta(a-\eps\tilde{z})} \sigma(x_m-\tau \omega(a-\eps\tilde{z}), \omega(a-\eps\tilde{z})) \rd\tau}\\
	&\hspace{1cm}K(x_m-\zeta(a-\eps\tilde{z})\omega(a-\eps\tilde{z}),\omega(a-\eps\tilde{z}),v_i)\cdot\\
	& \hspace{1cm}e^{-\int_0^{t_m-\zeta(a-\eps\tilde{z})} \sigma(x_m-\zeta(a-\eps\tilde{z})\omega(a-\eps\tilde{z})-v_i\tau,v_i)\,  \rd\tau}\phi_x\left(\tilde{z}\right)\\
	&\hspace{1cm}\frac{\hat{s}^2}{\zeta(a-\eps\tilde{z})^2}\frac{1-\< v_i,\hat{v}\>}{1-\<v_i,\omega(a-\eps\tilde{z})\>}\,\rd\tilde{z}
\end{align*}
where we used the determinent of the Jacobian of $\mathcal{S}$ being $s^2 (1-\<v,v_i\>)$, and the substitution $\tilde{z}=\frac{a-z}{\eps}$ for $a:= x_m-x_o$ in the last step. For a visualization of the quantities, see Figure~\ref{fig:quantities}.  The fact that a small  ball  around $a$ with radius of order $\eps^{\alpha}$ is contained in $\mathcal{S}(U)$ for every $\eps$ ensure that $\frac{a-\mathcal{S}(U)}{\eps}$  will eventually contain the full support $B(0,1)$ of $\phi_x$ for small $\eps$, see  Fig~\ref{fig:Stransform}.  Together with the continuity of $K, \sigma,\zeta, \omega$, this allows the application of the dominated convergence theorem:
\begin{align*}
	\lim_{\delta,\nu,\eta\to 0}M_{\psi}(\rho_1)\xrightarrow[]{\eps\to 0}& K(x_i,\hat{v},v_i)\int_{B(0,1)}\phi_x(\tilde{z}) \, \rd\tilde{z} =K(x_i,\hat{v},v_i),
\end{align*}
where we used the form \eqref{eq:ZetaOmega} of $\zeta, \omega$ to see $\zeta(a-\eps \tilde{z}) /\hat{s}\to  1$ while $\omega(a-\eps\tilde{z}) \to \hat{v}$.
\end{proof}
\begin{remark}
The proof for all three lemmas are local-in-time, in the sense that the measurement time is converging to $0$. This means that we can easily extend the result to deal with time-dependent $K$ as well. 
Suppose $K(x,t,v,v')$ 
should be recovered for a specific $t$ value, then a new experiment is started at time $t$, meaning both the initial data $\phi$ and the measurements $\psi$ should be prepared at reference time $t$.
\end{remark}

\begin{figure}[H]
\centering
\includegraphics[height = 4cm]{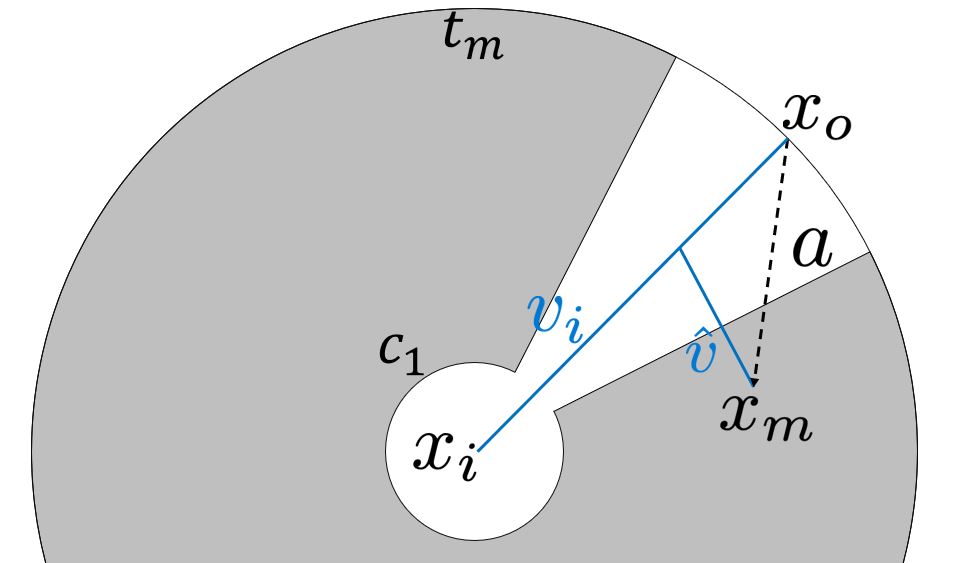}
\caption{Geometry and quantities used in the proof, displayed in 2D. In this figure, $x_o = x_i + v_it_m$, the location of the particle assuming it does not tumble, see definition~\eqref{eqn:x_o}. Note that $t_m$ is the length between $x_i$ and $x_o$. The gray area is $x_i+vs$ for $(s,v)\in U$. This is  the annulus $A$ in figure \ref{sfig:annulus} translated by $x_i$. The red point $\hat{v}\hat{s}$ is not depicted.}
\label{fig:quantities}
\end{figure}



\begin{figure}[H]
\centering
\begin{subfigure}{0.45\textwidth}
	\centering
	\includegraphics[width = \textwidth]{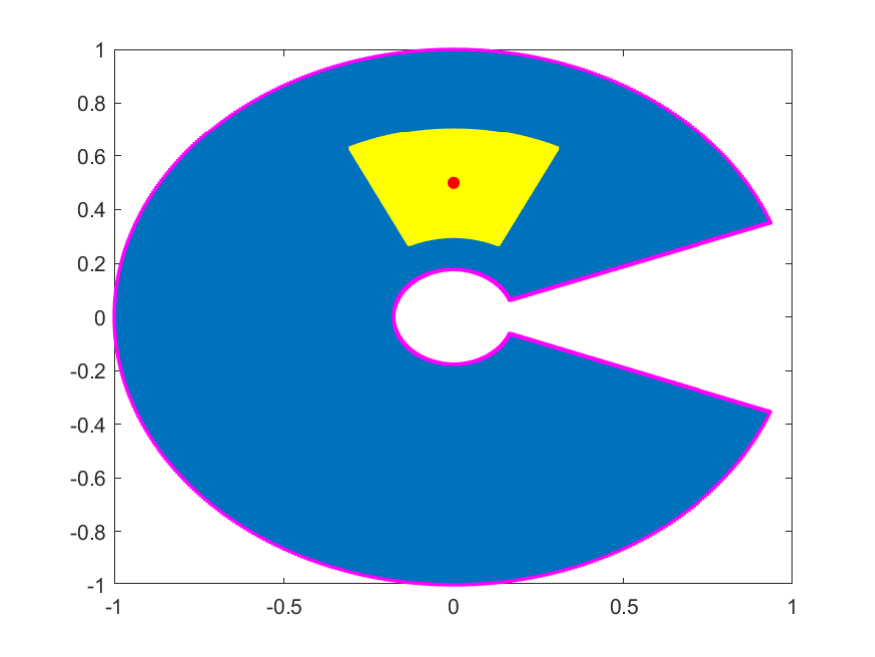}
	\caption{Sliced annulus $A$.}
	\label{sfig:annulus}
\end{subfigure}~
\begin{subfigure}{0.45\textwidth}
	\centering
	\includegraphics[width = \textwidth]{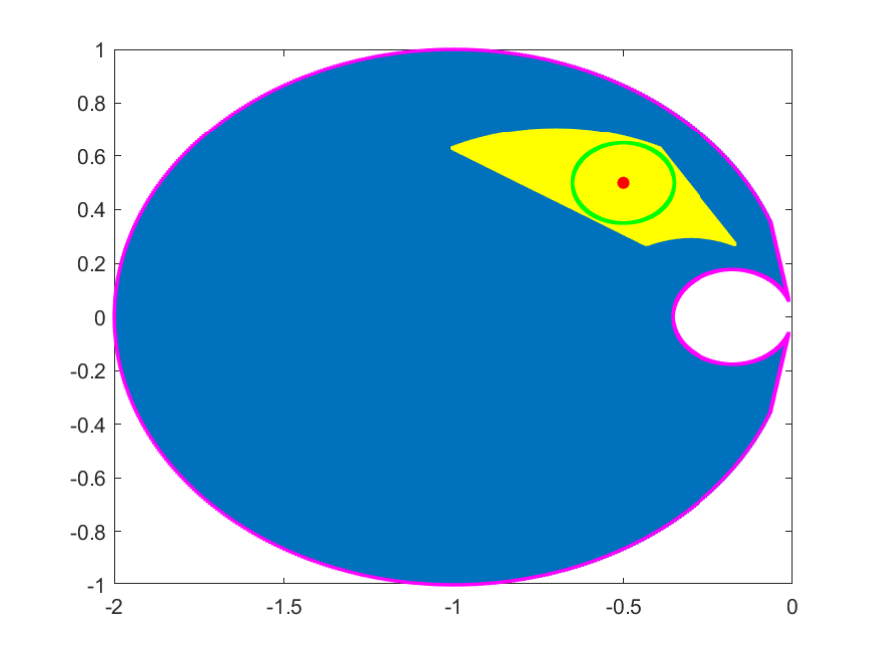}
	\caption{Image of $\mathcal{S}(U).$}
\end{subfigure}
\caption{Perturbation of U by $\mathcal{S}$ in 2D. In a first step, $A:=\{vs\mid (v,s)\in U\}$ is displayed. The red dot marks $\hat{v}\hat{s}$ which is bounded away from the boundaries of $A$ by construction. The yellow slice of an annulus is a neighbourhood of $\hat{v}\hat{s}$ that is bounded by the  arches of two circles. The image of $\mathcal{S}(U)$ is obtained by shifting each point in $vs\in A$ by $-v_is$. In this picture, the red dot is $a=\mathcal{S}(\hat{s},\hat{v})$. The image of the yellow area is bounded by the same arches of the circles, but the circles were shifted in direction $-v_i$ such that they touch $0$. One can choose the yellow slice of the annulus large enough such that a ball with radius  of order $ \eps^{\alpha}$ - whose boundary is depicted in green - is contained in the yellow image area.  }
\label{fig:Stransform}
\end{figure}


\newpage\newpage

\section{Conclusions}
\label{sec:conclusions}

In this paper we work on a classical kinetic chemotaxis model, and give a rigorous proof for using density measurement to reconstruct tumbling and damping coefficient. As stated in the introduction, chemotaxis is a heavily studied research area and there are abundant models. What we consider in~\eqref{eq:chemotaxis} is one of the most classical ones that was derived from the study of biased random walks~\cite{AltChemotaxis}. We show that when given a special design of initial data, the population density, one specific macroscopic quantity as a function of time, contains sufficient information to recover the microscopic quantities, such as velocity tumbling kernel and its associated damping coefficient.

This leaves various directions  unexplored, of which we list a few here:
\begin{itemize}
\item{Result-wise:} In the current paper the stability  of the reconstruction is not addressed. Indeed, Lemma~\ref{lem:f1} suggests that the reconstruction of $\sigma$ requires differentiating the data. We thus expect bad stability of this reconstruction in $L_\infty$ norm if the data is also assumed to be in $C$ in time. To obtain good stability, a proper norm  higher than $C^1$ needs to be selected. As it is not trivial, how the details are involved, we leave the discussion for stability for future work.
\item{Model wise:} we only showcase unique reconstruction for a very specific setting. More complicated models, of which some are listed below,  are not yet considered.
\begin{itemize}
	\item In our model, the space distribution of stimuli is fixed. This can be viewed as  a simplification, as in practice, bacteria interact with the environment, and may consume or produce the stimuli, for instance in self-attraction or self-repulsion phenomena. This leads to   changes in the chemical concentration which can in turn  generate   interesting patterns, see \cite{Perthame_PatternFormationKineticChemotaxis_2018}  or \cite{Painter_PatternFormation} and references therein. On the mathematical side, this case is treated by  a coupled system of  the chemotaxis equation \eqref{eq:chemotaxis} and an elliptic or parabolic equation for the chemical signal~\cite{Bournaveas_BlowUpKineticChemotaxis,BournaveasPoissonModel,chalub2004kinetic, chalub2006PreventOvercrowdingKineticChemotaxis}. However, it is practically almost impossible to trace the bacteria trajectory and measure the time dynamics of chemical concentration simultaneously. This prevents the quantification of most chemotaxis models except for some tightly controlled case  \cite{ChemotaxisRecoverDGamma,giometto2015generalized,ChemotaxisMeasureNoChemoattractant} 
	or well studied species \cite{Li_BarrierCrossingEcoliDetection, Si2012InternalVar}.  The study of the potential to extend thus presented techniques to the above mentioned more complicated settings could lead to interesting and practically relevant results. 
	\item Literature also provides more sophisticated kinetic chemotaxis models, for example, models that incorporate birth/death processes~\cite{OthmerAltDispersalModels,OthmerHillen2Chemo}, the tumbling time~\cite{Schmeiser_noninstantaneousTumbling_2022}, or the adaptation process with internal variables~\cite{ErbanOthmerKineticChemotaxis,MinTang2014InternalVar,XueOthmer_kineticModelInternalVar_2009,TangXue_FractionalDiffusionLimit_2021}.
	It would be interesting to investigate whether macroscopic quantities can provide enough information to recover the microscopic parameters for these more sophisticated kinetic models. 
\end{itemize}
\end{itemize}

Despite its obvious significance, inverse problem in math bio is still at its infancy. Many related problems are left unaddressed. In the framework of kinetic formulation for bacteria-motion, singular decomposition technique has demonstrated its flexibility and is very compatible with the kinetic formulation in the inverse problem setting. We expect to further investigate various problems listed above along this direction.

\appendix
\section*{Acknowledgments}
We would like to thank Marlies Pirner for the inspiring discussions that influenced the design of Lemma \ref{lem:estMfgeqN}. We also thank Beno\^it Perthame for insightful comments regarding this work.



\vspace{6pt} 



\subsection*{Funding}
K.H. acknowledges support by the \emph{W\"urzburg Mathematics Center for Communication and Interaction} (WMCCI) as well as the  \emph{German Academic Scholarship Foundatio} (Studienstiftung des deutschen Volkes) and the \emph{Marianne-Plehn-Program}. The work of Q.L. is supported in part by NSF-DMS-1750488 and ONR-N00014-21-1-2140, and Vilas Young Investigator award. M. Tang is supported by NSFC 11871340, NSFC12031013.
%


\end{document}